\documentclass[12pt,a4paper]{article}

\usepackage{url,natbib}

\usepackage{graphicx,amsmath,amsfonts,defns,moreverb,microtype}
\allowdisplaybreaks
\usepackage{pgfplots}
\pgfplotsset{compat=newest}
\newcommand{\tikzsetnextfilename}[1]{}

\usepackage{amsthm}
\theoremstyle{theorem}
\newtheorem{conjecture}[theorem]{Conjecture}

\atBegin{theorem}{\vspace{2ex plus 1ex minus 1ex}}
\atBegin{lemma}{\vspace{2ex plus 1ex minus 1ex}}

\def\ajrSplit#1#2\ajrEndSplit{\def\ajrcar{#1}\def\ajrcdr{#2}}
\def\Vec{}
\renewcommand{\Vec}[1]{%
  \typeout{**** Command  \ #1v denotes vector #1}%
  \expandafter\ajrSplit#1\ajrEndSplit%
  \ifx\ajrcdr\empty 
    \expandafter\def\csname#1v\endcsname%
    {\ensuremath{\vec #1}}%
  \else 
    \expandafter\def\csname#1v\endcsname%
    {\ensuremath{\vec{\csname#1\endcsname}}}%
  \fi%
  }
\Vec U \Vec g \Vec v \Vec u \Vec s

\newcommand{\Bb}[1]{%
  \expandafter\def\csname#1#1\endcsname%
  {\ensuremath{\mathbb #1}}}
\Bb X\Bb T\Bb R\Bb H\Bb I\Bb J\Bb L \Bb E

\newcommand{\Cal}[1]{%
  \expandafter\def\csname c#1\endcsname%
  {\ensuremath{\mathcal #1}}}
\Cal E \Cal I \Cal L \Cal M \Cal N \Cal R

\newcommand{\secref}[1]{\ref{sec:#1}}

\title{Smooth subgrid fields underpin rigorous closure in spatial discretisation of reaction-advection-diffusion \pde{}s}
\author{G.~A. Jarrad\thanks{\protect\url{mailto:geoff.jarrad@adelaide.edu.au}}
\and 
A.~J. Roberts\thanks{School of Mathematical Sciences, University of Adelaide, South Australia~5005, Australia.  
\protect\url{mailto:anthony.roberts@adelaide.edu.au}}
}
\begin{document}
\maketitle

\begin{abstract}
Finite difference\slash element\slash volume methods of discretising \pde{}s impose a subgrid scale interpolation on the dynamics. 
In contrast, the holistic discretisation approach developed herein constructs a natural subgrid scale field adapted to the whole system out-of-equilibrium dynamics.
Consequently, the macroscale discretisation is fully informed by the underlying microscale dynamics. 
We establish a new proof that in principle there exists an exact closure of the dynamics of a general class of reaction-advection-diffusion \pde{}s, and show how our approach constructs new systematic approximations to the in-principle closure starting from a simple, piecewise-linear, continuous approximation.
Under  inter-element coupling conditions that guarantee continuity of several field properties, 
the holistic discretisation possesses desirable properties such as a natural cubic spline first-order approximation to the field, and the self-adjointness of the diffusion operator under periodic, Dirichlet and Neumann macroscale boundary conditions.
As a concrete example, we demonstrate the holistic discretisation procedure on the well-known Burgers' \pde{},
and compare the theoretical and numerical stability of the resulting discretisation to other approximations.
The approach developed here promises to be able to systematically construct automatically good, macroscale discretisations to a wide range of \pde{}s, including wave \pde{}s. 
\end{abstract}

\section{Introduction}\label{sec:intro}

This article's scope is the accurate and stable spatial discretisation of nonlinear \pde{}s for a field~\(u(x,t)\) satisfying reaction-advection-diffusion \pde{}s in the general form
\begin{equation}
u_t=F(u_x)_x+\alpha G(x,u,u_x)
\label{eq:genpde}
\end{equation}
for suitably smooth functions~\(F\) and~\(G\), and \(F\)~strictly monotonic increasing,
where subscripts \(x\) and~\(t\) denote spatial and temporal derivatives, respectively. 
Although most of this article addresses \pde~\eqref{eq:genpde}, Section~\ref{sec:smwpde} discusses generalising the theoretical support to wave-like \pde{}s obtained by replacing~\(u_t\) by~\(u_{tt}\) in~\eqref{eq:genpde}.
Given \(N+1\)~discrete points in 1D space, \(x=X_j\) for \(j\in\JJ=\{0,1,\ldots,N\}\)\,, we define grid values \(U_j(t)=u(X_j,t)\)\,.
Then the aims are to use centre manifold theory \cite[e.g.]{Carr81} to (\S\ref{sec:lin}): firstly, establish a new proof that in principle there exists an \emph{exact} closure of the dynamics of the \pde~\eqref{eq:genpde} in terms of these grid values, \(d\Uv/dt=\gv(\Uv)\); secondly, establish that such a closure is emergent from general initial conditions; and thirdly, show how to construct new systematic approximations to the in-principle closure.
This new theory is applied in Sections~\ref{sec:nonlin} and~\ref{sec:numeric} to construct and evaluate the new approach for the classic example of the nonlinear advection--diffusion Burgers' \pde
\begin{eqnarray}
	\D tu = \nu\DD xu - \alpha u\D xu\,.
\label{eq:burgers}
\end{eqnarray}
Generalisation of the approach to two or more spatial dimensions remains for further research but should be analogous to that established by \cite*{Roberts2011a}.

The spatial domain~\(\XX\) is of length~\(L\), \(0\leq x\leq L\)\,, and we mostly restrict attention to solutions~\(u(x,t)\) which are \(L\)-periodic in space, but occasionally comment on the cases of homogeneous Dirichlet boundary conditions, \(u(0,t)=u(L,t)=0\)\,, and Neumann boundary conditions, \(u_x(0,t)=u_x(L,t)=0\)\,.
The first step is to partition~$\XX$ into $N$~equi-spaced intervals bounded by the $N+1$ grid-points~$X_j$ with spacing~\(H\).
Traditional spatial discretisation of such \pde{}s, whether finite difference, finite element, or finite volume, imposes assumed fields in each element and then derives approximate rules for the evolution in time of the parameters of the imposed fit.
Our dynamical systems (holistic) approach is to let the \pde~\eqref{eq:genpde} determine the subgrid fields in order to remain faithful to the \pde, as demonstrated explicitly for Burgers' \pde~\eqref{eq:burgers}.
The multiscale derivation of the so-called stabilized schemes \cite[e.g.]{Hughes95} appears analogous to the first step of the construction described by Section~\ref{sec:nonlin}.
A previous dynamical systems approach constructs subgrid fields by systematically refining a piecewise constant initial approximation \cite[e.g.]{Roberts98a, Roberts00a, Roberts2011a}---an approach that adapts to the multi-scale gap-tooth scheme \cite[e.g.]{Roberts06d, Kevrekidis09a}.
The new approach here systematically refines a continuous piecewise linear initial approximation with the aim of more accurately encoding subgrid scale effects in the macroscale closure.

We aim for the dynamics of the field~$u(x,t)$ to be summarised by the macroscale coarse variables ${\Uv}=(U_0,U_1,\ldots,U_N)$, 
where we choose these coarse variables to be the grid values 
\begin{equation}
U_j(t):=u(X_j,t) \quad\text{for all } j\in\JJ,t\in \TT.
\label{eq:gridval}
\end{equation}
Henceforth we assume $U_0=U_N$ due to the imposed periodicity, unless otherwise stated.
New theory developed in Section~\ref{sec:lin} asserts that in principle an \emph{exact} closure exists (a slow manifold); that is, there is some system of \ode{}s 
\begin{eqnarray}
	\dot U_j = g_j({\Uv}) \quad\text{for all }j\in\JJ\,,
\label{eq:temporal}
\end{eqnarray}
that gives exact solutions of the \pde.
A traditional approach is to use centred approximations:
\begin{eqnarray*}
\dot U_j&\approx& -\alpha\frac1{2H}U_j(U_{j+1}-U_{j-1})+\nu\frac1{H^2}(U_{j+1}-2U_j+U_{j-1})
\\&&{}=-\alpha U_j\mu\delta U_j/H+\nu\delta^2 U_j/H^2,
\end{eqnarray*}
for centred difference $\delta=\sigma^{1/2}-\sigma^{-1/2}$,
centred mean $\mu=(\sigma^{1/2}+\sigma^{-1/2})/2$,
and shift operator $\sigma U_j=U_{j+1}$.
However, the nonlinear advection term has another plausible representation, namely the conservative form~$\mu\delta (U_j^2)/2H$.
For illustrative purposes, Section~\ref{sec:nonlin} compares results with Burgers' \pde~\eqref{eq:burgers} discretised to the so-called mixture model
\begin{equation}
	\dot{U}_j = 
	-(1-\theta)\alpha\frac{U_j\mu\delta U_j}{H}
    -\theta\alpha\frac{\mu\delta (U_j^2)}{2H}
    +\nu\frac{\delta^2 U_j}{H^2}\,.
\label{eq:mixture}
\end{equation}
In contrast, Section~\secref{nonlin} shows our holistic approach has no such representational ambiguity, and constructs at first-order the specific model
\begin{equation}
	\dot{U}_j = S\left[
	-\alpha\frac{U_j\mu\delta U_j}{3H}
    -\alpha\frac{\mu\delta (U_j^2)}{3H}
    +\nu\frac{\delta^2 U_j}{H^2} \right],
\label{eq:holistic1}
\end{equation}
for nonlocal operator $S=(1+\delta^2/6)^{-1}$. 
Apart from the operator~$S$,
this holistic model matches the mixture model~\eqref{eq:mixture} for parameter $\theta=\frac{2}{3}$.
This parameter value is exactly the critical value shown by \cite{Fornberg73} to be necessary for stable simulation (with $\nu=0$ and $\alpha=1$) for a selection of numerical integration schemes.
Section~\secref{numeric} further compares the numerical behaviour of our holistic and established mixture models.

A crucial part of the new methodology is to express the physical field~\(u(x,t)\) naturally in terms of the coarse variables~${\Uv}(t)$ for out-of-equilibrium dynamics. 
That is, as illustrated by the two approximate examples of Figure~\ref{fig:uhat}, we construct the field (a slow manifold)
\begin{eqnarray}
	u(x,t)  = {u}(x,{\Uv(t)}),
	\label{eq:spatial}
\end{eqnarray}
where the time evolution of the field~\(u\) occurs via the evolving coarse variables~\(\Uv(t)\).
Whether the symbol~\(u\) denotes~\(u(x,t)\) or~\(u(x,\Uv)\) should be clear from the context.
The complete holistic framework comprises equations~\eqref{eq:temporal} and~\eqref{eq:spatial},  in conjunction with suitable boundary 
and  inter-element coupling conditions to be specified  
in more detail in Section~\secref{lin}. 

In particular, the Rayleigh--Ritz theorem motivates coupling conditions that give a piecewise linear function as the leading approximation (e.g., the blue~\(u^0\) of Figure~\ref{fig:uhat}).
Approximately constructing a slow manifold is analogous to estimating eigenvalues of a perturbed matrix.
For a self-adjoint operator~\cL, the Rayleigh--Ritz theorem is that an approximate eigenvector~\vv, with error~\Ord{\epsilon},
predicts a corresponding eigenvalue $\lambda = {\langle{\vv},{\cL}{\vv}\rangle}/{\|{\vv}\|^2}$ with asymptotically smaller error~\Ord{\epsilon^2}.
This suggests that the more accurate we make an initial approximation to the field~\(u\), the more accurate the predicted evolution on the slow manifold.
Consequently, this article develops a systematic approximation to an in-principle exact discrete closure based upon the novel approach of systematically refining a piecewise linear and continuous subspace approximation to the field~\(u\).

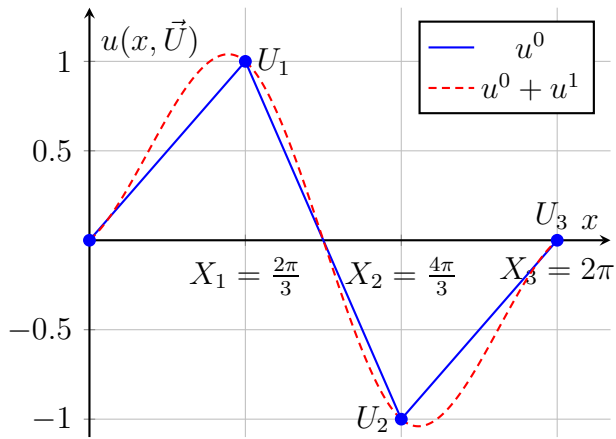
\begin{figure}
\centering
\tikzsetnextfilename{smoothInterp}
\begin{tikzpicture}[declare function = { 
	xi1(\x) = \x/2.0944; xi2(\x) = \x/2.0944-1; xi3(\x) = \x/2.0944-2; 
	u0(\x) = 0; u1(\x) = 1; u2(\x) = -1; u3(\x) = 0;
	sd2u0(\x) = -4.5*u0(\x)+3*u1(\x)-1.5*u2(\x)+3*u3(\x);
	sd2u1(\x) = -4.5*u1(\x)+3*u2(\x)-1.5*u3(\x)+3*u0(\x);
	sd2u2(\x) = -4.5*u2(\x)+3*u3(\x)-1.5*u0(\x)+3*u1(\x);
	sd2u3(\x) = -4.5*u3(\x)+3*u0(\x)-1.5*u1(\x)+3*u2(\x);
}]
\begin{axis}[ xlabel={$x$}, ylabel={$u(x,\Uv)$}
  ,axis x line=middle,axis y line=middle
  ,thick,grid,samples=101
  ,ymin=-1.1,ymax=1.3 ,domain=0:6.2832,xmax=7
  ,xtick={2.0944,4.1888,6.2832}
  ,xticklabels={$X_1=\frac{2\pi}{3}$,$X_2=\frac{4\pi}{3}$,$X_3=2\pi$}
  ,legend pos=north east
  ] 
\addplot [blue,no marks,domain=0:2.0944] {xi1(x)*u1(x)+(1-xi1(x))*u0(x)}; 
\addlegendentry{${u^0}$};
\addplot [densely dashed,red,no marks,domain=0:2.0944] {xi1(x)*u1(x)+(1-xi1(x))*u0(x)
	+1/6*(xi1(x)^3*sd2u1(x)+(1-xi1(x))^3*sd2u0(x)-xi1(x)*(sd2u1(x)-sd2u0(x))-sd2u0(x))
}; 
\addlegendentry{${u^0}+{u^1}$};
\addplot [blue,no marks,domain=2.0944:4.1888] {xi2(x)*u2(x)+(1-xi2(x))*u1(x)}; 
\addplot [blue,no marks,domain=4.1888:6.2832] {xi3(x)*u3(x)+(1-xi3(x))*u2(x)}; 
\addplot [densely dashed,red,no marks,domain=2.0944:4.1888] {xi2(x)*u2(x)+(1-xi2(x))*u1(x)
	+1/6*(xi2(x)^3*sd2u2(x)+(1-xi2(x))^3*sd2u1(x)-xi2(x)*(sd2u2(x)-sd2u1(x))-sd2u1(x))
}; 
\addplot [densely dashed,red,no marks,domain=4.1888:6.2832] {xi3(x)*u3(x)+(1-xi3(x))*u2(x)
	+1/6*(xi3(x)^3*sd2u3(x)+(1-xi3(x))^3*sd2u2(x)-xi3(x)*(sd2u3(x)-sd2u2(x))-sd2u2(x))
}; 
\addplot [blue,only marks,mark=*] coordinates 
{(0,0)(2.0944,1)(4.1888,-1)(6.2832,0)};
\node[right] at (axis cs:2.0944,1) {$U_1$};
\node[left] at (axis cs: 4.1888,-1) {$U_2$};
\node[above] at (axis cs: 6.2832,0) {$U_3\ $};
\end{axis}
\end{tikzpicture}
\caption{An example of the smooth subgrid field provided by the holistic discretisation process,
where the piecewise-linear initial approximation~${u^0}(x,\Uv)$ is smoothed by the first-order correction~${u^1}(x,\Uv)$  (for nonlinearity $\alpha=0$).
The correction~\(u^1\) forms a cubic spline; however, it is derived directly from the \pde\ itself, rather than obtained by 
imposing such an interpolation.}
\label{fig:uhat}
\end{figure}

\section{An example introduces theory and method}
\label{sec:beitm}

As an introduction to the methodology and theory, this section investigates the modelling of Burgers' \pde~\eqref{eq:burgers} on the specific domain \(-1<x<1\)\,, with basic Dirichlet boundary conditions that \(u(\pm 1,t)=0\)\,, and with viscosity \(\nu=1\) for definiteness.
For introductory simplicity, the domain space is partitioned into just two intervals, \(-1<x<0\) and \(0<x<1\).
Our aim is to model the dynamics of the whole field~\(u(x,t)\) by simply the dynamics of the grid value \(U(t):=u(0,t)\) of the field at the single, central, interior grid-point \(X=0\).

The dynamics in the two intervals need to be coupled to each other to form a solution valid over the whole domain.
Conventional numerical methods \emph{impose} an assumed interpolation field and then derive a corresponding model.
In contrast, here we craft a coupling that moderates the communication between the two intervals, and then let the \pde~\eqref{eq:burgers} itself tell us the appropriate out-of-equilibrium fields and model.
The desired full coupling between the two intervals is of \(C^1\)~continuity: \([u]=[u_x]=0\) where we introduce~\([\cdot]\) to denote the jump in value across the grid-point \(X=0\); that is, \([u]=u|_{0^+}-u|_{0^-}\)\,.
For reasons developed below, we embed Burgers' \pde~\eqref{eq:burgers} in a family of problems with the moderated coupling between intervals of
\begin{equation}
[u]=0\quad\text{and}\quad
[u_x]+2(1-\gamma)u=0 \quad\text{at }x=X=0\,;
\label{eq:cc1}
\end{equation}
that is, the field is continuous but the derivative has a discontinuity depending upon homotopy  parameter~\(\gamma\) (corresponding to the general case~\eqref{eqs:cc}).  
We derive below that \(\gamma=0\) provides a useful base to apply powerful centre manifold theory.
When \(\gamma=1\), the coupling~\eqref{eq:cc1} reverts to requiring \(C^1\)~continuity across \(x=0\) to restore the \pde\ over the entire spatial domain.

To show there is a useful (slow) centre manifold, we start with equilibria in the system (corresponding to Lemma~\ref{thm:equil}, p.\pageref{thm:equil}).
The  \pde~\eqref{eq:burgers}, with diffusivity \(\nu=1\), together with coupling conditions~\eqref{eq:cc1}, and the Dirichlet boundary conditions, has a subspace~\EE\ of equilibria: for each~\(U\),
\begin{equation}
u=(1-|x|)U\quad\text{and }\gamma=\alpha=0\,.
\label{eq:equil}
\end{equation}
The spectrum about each of these equilibria determine the manifold structure (corresponding to Lemmas~\ref{thm:dicho} and~\ref{lem:hs}, pp.\pageref{thm:dicho},\pageref{lem:hs}).
We seek linearised solutions \(u(x,t)\approx (1-|x|)U+e^{\lambda t}v(x)\) for small~\(v\):  the diffusion \pde~\eqref{eq:burgers} becomes the eigenproblem
\begin{equation}
-v_{xx}+\lambda v=0\,,
\quad\text{such that }[v]=[v_x]+2v=0\text{ at }x=0\,,
\label{eq:lindiff}
\end{equation}
with homogeneous Dirichlet boundary conditions \(v(\pm 1)=0\)\,.
\begin{itemize}
\item Corresponding to eigenvalue \(\lambda=0\) is the neutral solution \(v\propto1-|x|\) reflecting the direction of the subspace~\EE\ of equilibria.
\begin{figure}
\centering
\begin{tabular}{cc}
\tikzsetnextfilename{simpevec}
\begin{tikzpicture}
\begin{axis}[xlabel={$x$},ylabel={$v(x)$}
,domain=-1:1,axis x line=middle, axis y line=middle
,ymax=1.2,xmin=-1.1,xmax=1.2,thick,smooth]
    \addplot+[dashed,no marks] {sin(180*x)};
    \addplot+[thin,no marks] {sin(4.4934*180/3.1416*(1-abs(x))};
    \addplot+[densely dotted,no marks] {sin(360*x)};
    \addplot+[thick,no marks] {sin(7.7253*180/3.1416*(1-abs(x))};
\end{axis}
\end{tikzpicture}
&
\parbox[b]{0.45\linewidth}{\caption{\label{fig:simpevec}\sloppy%
Eigenfunctions~\(v(x)\) of the linearised problem~\eqref{eq:lindiff} corresponding to negative eigenvalues: blue-dashed,~\(-\pi^2\); red-thin,~\(-20.191\); brown-dotted,~\(-4\pi^2\); and black-thick,~\(-59.680\).}}
\end{tabular}
\end{figure}
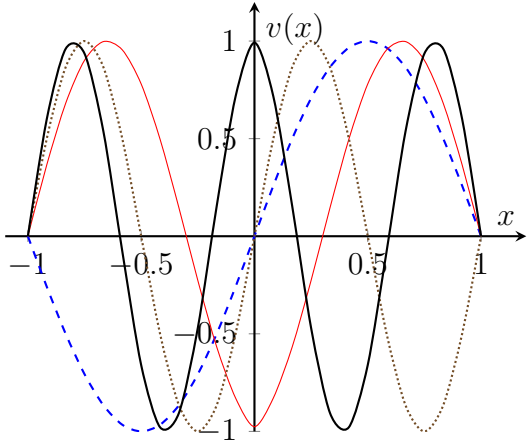

\item Some negative eigenvalues \(\lambda=-k^2\) correspond to eigenfunctions of the form \(v\propto\sin[k(1-|x|)]\).
These arise by necessity from the \pde, the homogeneous Dirichlet boundary conditions, and the continuity of~\(v\).
By straightforward algebra, the jump in the derivative determines the wavenumbers~\(k\) from the solutions of \(k=\tan k\)\,, namely the wavenumbers \(k=4.4934,7.7253,10.9041,\ldots\)\,.
That is,  non-zero eigenvalues of the linearised problem are \(\lambda=-20.191,-59.680,-118.900,\ldots\)\,.
Figure~\ref{fig:simpevec} plots (solid) the corresponding eigenfunctions for the two smallest magnitude of these eigenvalues.

\item Negative eigenvalues also arise from eigenfunctions of the form \(v\propto\sin(kx)\).  
The boundary and coupling conditions determine the wavenumbers \(k=n\pi\) for \(n=1,2,3,\ldots\)\,.
That is, the other non-zero eigenvalues are \(\lambda=-\pi^2, -4\pi^2, -9\pi^2, \ldots\)\,.
Figure~\ref{fig:simpevec} plots (dashed) the corresponding eigenfunctions for the two smallest magnitude of these eigenvalues.
\end{itemize}
One of the beautiful properties of the coupling conditions~\eqref{eq:cc1} is that with them the diffusion operator~\(\DD x{}\) is self-adjoint (analogous to Lemma~\ref{thm:sa}, p.\pageref{thm:sa}).
Hence there are only real eigenvalues of the linear problem~\eqref{eq:lindiff}, namely the ones found above.
To confirm self-adjointness under the usual inner product, \(\left<u,v\right>=\int_{-1}^1 u(x)v(x)\,dx\)\,,
consider
\begin{eqnarray*}
\left<u,v_{xx}\right>
&=&\int_{-1}^1 uv_{xx}\,dx
\quad(\text{then using integration by parts})
\\&=&\left[uv_{x}-vu_{x}\right]_{-1}^{0^-}
+\left[uv_{x}-vu_{x}\right]_{0^+}^1+\int_{-1}^1 u_{xx} v\,dx
\\&&\quad(\text{using the Dirichlet boundary conditions})
\\&=&-\left[uv_{x}-vu_{x}\right]_{0^-}^{0^+}+\left<u_{xx},v\right>
\\&&\quad(\text{using continuity at }x=0)
\\&=&-u|_0[v_{x}]+v|_0[u_{x}]+\left<u_{xx},v\right>
\\&&\quad(\text{using the jump in derivative at }x=0)
\\&=&u|_02(1-\gamma)v|_0-v|_02(1-\gamma)u|_0+\left<u_{xx},v\right>
\\&=&\left<u_{xx},v\right>.
\end{eqnarray*}
This useful self-adjointness is not a property of previous holistic discretisations \cite[Part~V, e.g.]{Roberts2014a}, but is a new feature established by the new approach developed herein.

Because the spectrum consists of a zero eigenvalue and all the rest negative (\(\leq -\pi^2<-9\)), centre manifold theory \cite[e.g.]{Carr81} 
assures us that there exists a slow manifold in some neighbourhood of the subspace~\EE\ of equilibria (corresponding to Theorem~\ref{thm:gsm}); that is, global in amplitude~\(U\) and local in parameters~\(\gamma\) and~\(\alpha\).
Also, the theory guarantees that all solutions in the neighbourhood are attracted exponentially quickly, at least as fast as roughly~\(e^{-9t}\), to solutions on the slow manifold.
That is, the slow manifold and the evolution thereon emerges from general initial conditions.

A theorem \cite[Thm.~6.10, e.g.]{Carr81} also guarantees that when we approximate the slow manifold and its evolution to a residual of~\Ord{\gamma^p}, the slow manifold and its evolution are correct to errors~\Ord{\gamma^p}.
By straightforward machinations not detailed here~\cite[Ch.~14]{Roberts96a, Roberts2014a} we arrive at the expressions that the slow manifold and the evolution thereon are
\begin{equation}
u\approx\left[1-|x|+\gamma(|x|-\tfrac32x^2+\tfrac12|x|^3)\right]U
\quad\text{such that }\dot U\approx-3\gamma U\,.
\label{eq:egsm1}
\end{equation}
Substituting these expressions into the heat \pde~\eqref{eq:burgers} (\(\alpha=0\)), with the boundary and coupling conditions~\eqref{eq:cc1} we find the equations are satisfied to residual~\Ord{\gamma^2} and so the approximation theorem asserts these expressions are approximations with errors~\Ord{\gamma^2}.

\begin{figure}
\centering
\begin{tabular}{cc}
\tikzsetnextfilename{simpleshape}
\begin{tikzpicture}
\begin{axis}[small,xlabel={$x$},ylabel={$u(x,U)$}
,domain=-1:1,axis x line=middle, axis y line=middle
,ymax=1.1,xmin=-1.1,xmax=1.2,thick]
    \addplot+[no marks] {1-abs(x)};
    \addplot+[dotted,no marks] {1-1.5*x^2+abs(x)^3/2};
    \addplot+[no marks] {cos(90*x)};
\end{axis}
\end{tikzpicture}
&
\parbox[b]{0.55\linewidth}{\caption{\label{fig:simp}%
A comparison of approximations to the long-term, quasi-stationary, decay of the heat \pde: blue-solid, \(u\propto 1-|x|\) is the basic linear approximation~\eqref{eq:equil}; red-dotted, the derived cubic spline~\eqref{eq:egsm1} at full coupling \(\gamma=1\); and, almost indistinguishable, brown-solid, is the exact mode \(u\propto\cos(\pi x/2)\).
}}
\end{tabular}
\end{figure}
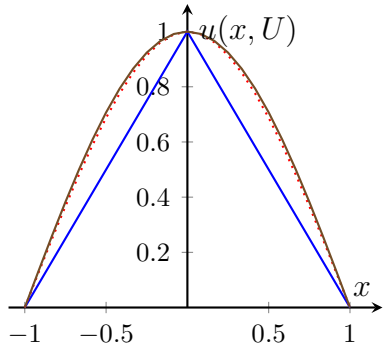

Although this approximation is based around parameter \(\gamma=0\)\,, we are interested in the physical value of the parameter \(\gamma=1\)\,.
Evaluating the slow manifold~\eqref{eq:egsm1} at \(\gamma=1\) gives 
\begin{equation*}
u\approx(1-\tfrac32x^2+\tfrac12|x|^3)U
\quad\text{such that }\dot U\approx-3 U\,.
\end{equation*}
The field~\(u\), plotted in Figure~\ref{fig:simp}, is an excellent cubic spline approximation to the correct~\(U\cos(\pi x/2)\) eigenfunction, also plotted in Figure~\ref{fig:simp}.
The predicted evolution \(U\propto e^{-3t}\) is a good approximation to the correct decay rate of~\(-\pi^2/4\)\,.

One key question is how can we be sure that evaluating at finite \(\gamma=1\) is within the finite neighbourhood of validity of the slow manifold?
Here computer algebra~\cite[Ch.~14]{Roberts96a, Roberts2014a} straightforwardly computes to high order to determine, for example, the slow evolution
\begin{equation*}
\dot U=-\big[3\gamma-0.6\gamma^2 +0.06857\gamma^3 -0.00128\gamma^5 +0.00008\gamma^6 +0.00004\gamma^7+\Ord{\gamma^8}\big]U.
\end{equation*}
Evidently the series in~\(\gamma\) appears to have a radius of convergence much larger than one.\footnote{Construction of the slow manifold to 40th~order in~\(\gamma\) (for \(\alpha=0\)) followed by a generalised Domb--Sykes plot \cite[Appendix]{Mercer90} predicts a convergence limiting singularity at \(\gamma_*=-0.9+i3.7\) (at an angle~\(103^\circ\) to the real \(\gamma\)-axis) indicating convergence for all \(|\gamma|<3.8\).}
Hence we predict that the neighbourhood of validity around~\EE\ includes the  case of interest, \(\gamma=1\)\,.

\begin{figure}
\centering
\begin{tabular}{cc}
\tikzsetnextfilename{nonlinsm}
\begin{tikzpicture}
\def\alfa{2}
\begin{axis}[xlabel={$x$},ylabel={$u(x,U)$}
,domain=-1:1,axis x line=middle, axis y line=middle
,xmin=-1.1,xmax=1.1,legend pos=north west,thick,smooth]
\foreach \uu in {2.0,1.5,1.0,0.5} {
    \addplot+[no marks] {\uu*(1-6/5*abs(x)^2-1/10*abs(x)^3+3/8*abs(x)^4-3/40*abs(x)^5)+\alfa*(\uu^2)*x*(2/5-abs(x)^2+3/4*abs(x)^3-3/20*abs(x)^4)+\alfa^2*(\uu)^3*(2/15*abs(x)^2-4/15*abs(x)^3+1/6*abs(x)^4-1/30*abs(x)^5)};
};
\end{axis}
\end{tikzpicture}
&
\parbox[b]{0.45\linewidth}{\caption{\label{fig:nonlinsm}%
Nonlinear slow manifold for Burgers' \pde~\eqref{eq:burgers} for viscosity \(\nu=1\) and nonlinearity \(\alpha=2\)\,.  
Drawn is the slow manifold~\(u(x,U)\) for representative amplitudes \(U=\tfrac12,1,\tfrac32,2\) to show the larger deformation at larger amplitudes.  
This approximation to the slow manifold was computed to errors~\Ord{\gamma^3+\alpha^3} and evaluated at full coupling \(\gamma=1\)\,.
}}
\end{tabular}
\end{figure}
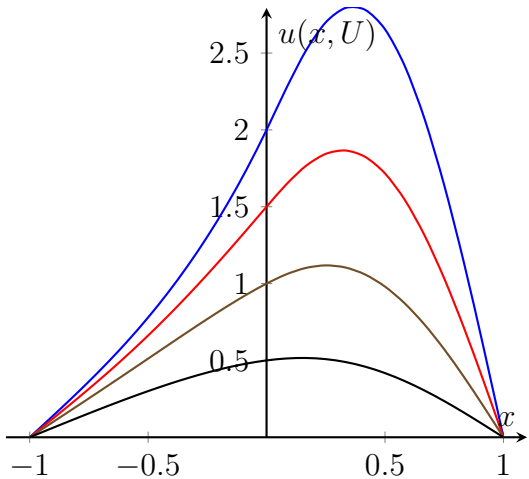

Centre manifold theory~\cite[Ch.~4, e.g.]{Carr81, Roberts2014a} was designed for nonlinear problems.
Thus it also applies here to the nonlinear Burgers' \pde~\eqref{eq:burgers} now with nonlinearity parametrised by~\(\alpha\) and similarly modelled with two intervals on the domain \(-1<x<1\)\,.
For example, modified computer algebra~\cite[Ch.~14]{Roberts96a, Roberts2014a} constructs the slow manifold plotted in Figure~\ref{fig:nonlinsm} on which the nonlinear evolution is
\begin{equation*}
\dot U=-(3\gamma+\tfrac35\gamma^2)U-\tfrac1{15}\gamma^2\alpha^2U^3
+\Ord{\gamma^3+\alpha^3}.
\end{equation*}
The nonlinear advection of Burgers' \pde\ generates steeper gradients in the subgrid field (Figure~\ref{fig:nonlinsm}) that enhance the decay as expressed by the cubic nonlinearity in this evolution equation for amplitude~\(U(t)\) \cite[cf.][\S5]{Hughes95}.

Key properties of this example are also exhibited in the application of the approach to the more general spatial discretisations discussed in subsequent sections: an analogous inter-element coupling engenders an emergent slow manifold; the linearised operator is self-adjoint; the first iteration constructs a cubic spline; and the resultant model at full coupling has attractive properties.

\section{Linearisation establishes the existence of a closure}\label{sec:lin}

We use centre manifold theory \cite[e.g.]{Carr81, Haragus2011} to establish (Theorem~\ref{thm:gsm}) the in-principle existence and emergence of a new exact closure to the dynamics of  \pde{}s in the class~\eqref{eq:genpde}.
Centre manifold theory is based upon an equilibrium or subspace of equilibria, and follows primarily from the persistence of a spectral gap in the spectrum of the linearised dynamics \cite[e.g.]{Roberts2014a}.

To find useful equilibria we embed the \pde~\eqref{eq:genpde} in a wider class of problems.
First partition the spatial domain into the \(N\)~intervals between the grid-points \(x=X_j\)\,: let the interval \(\XX_j=\{x\mid X_{j-1}<x<X_j\}\) and denote the punctured domain \(\tilde\XX:=\XX\backslash\{X_0,X_1,\ldots,X_N\}\).
For definiteness take the boundary conditions on the field~\(u(x,t)\) to be that it is \(L\)-periodic in space.
Then use \(u_j(x,t)\)~to denote solutions of the \pde~\eqref{eq:genpde} on the interval~\(\XX_j\), and reserve \(u(x,t)\), over~\XX\ or~\(\tilde\XX\) as appropriate, to denote the union over all intervals of such solutions.
To restore the original \pde~\eqref{eq:genpde} over the whole domain~\XX\ we couple the fields on each interval together.
By controlling the information flow between intervals we connect the original \pde\ over the whole domain to a useful base problem.
The general coupling conditions are  
\begin{subequations}\label{eqs:cc}%
\begin{eqnarray}
&&[u]_j=u|_{X_j^+}-u|_{X_j^-}=0
\quad\text{and }
\label{eq:ccc}
\\&&
[\nu u_x]_j=\frac{C(\gamma)}H\left[
(\nu u)|_{X_{j+1}^-} -(\nu u)|_{X_{j}^+}
+(\nu u)|_{X_{j-1}^+} -(\nu u)|_{X_{j}^-}\right],\qquad
\label{eq:ccj}
\end{eqnarray}
\end{subequations}
where coefficient \(\nu(x)=F'(u_x)\) is the effective diffusivity at each point in~\(\tilde\XX\), via the gradient~\(u_x\),
and where the factor~\(C(\gamma)\) is some smooth function such that \(C(0)=1\) and \(C(1)=0\) (typically \(C(\gamma):=1-\gamma\) as in~\eqref{eq:cc1}).

\begin{lemma}[equilibria] \label{thm:equil}
The \pde~\eqref{eq:genpde} on domain~\(\tilde\XX\) with coupling conditions~\eqref{eqs:cc} possesses an \(N\)-dimensional subspace~\EE\ of equilibria, parametrised by \(\Uv=(U_1,U_2,\ldots,U_N)\), for parameters \(\alpha=\gamma=0\)\,. 
Each equilibrium is of continuous, piecewise linear, fields~\(u^*(x)\) such that, on the \(j\)th~interval, the field 
\begin{equation}
u^{*}(x)=u_j^*(x)=(1-\xi_j)U_{j-1}+\xi_j U_j
\quad\text{where }\xi_j=(x-X_{j-1})/H
\label{eq:equil}
\end{equation}
is a local scaled space variable.
\end{lemma}

\begin{proof} 
With nonlinearity \(\alpha=0\) the \pde~\eqref{eq:genpde} takes the form \(u_t=F(u_x)_x\)\,.
For the piecewise linear field~\eqref{eq:equil}, the gradient \(u^*_{jx}=(U_j-U_{j-1})/H\) is constant on each~\(\XX_j\). 
Hence \(F(u^*_x)\) is constant on each~\(\XX_j\), and
consequently \(F(u^*_x)_x=0\) on~\(\tilde\XX\), giving an equilibria of the \pde\ on~\(\tilde\XX\).

From the field~\eqref{eq:equil},  \(u^*_j(X_j^-)=U_j=u^*_{j+1}(X_j^+)\) and hence \(u^*(x)\)~is continuous at~\(X_j\) to satisfy the coupling condition~\eqref{eq:ccc}.

Lastly, consider the condition~\eqref{eq:ccj} on the jump in the derivative.
For the field~\eqref{eq:equil}, the gradient is \(u^*_x=(U_j-U_{j-1})/H\) so, in terms of the constants 
\begin{equation}
\nu_j:=F'(u^*_{jx})=F'\big(\tfrac{U_j-U_{j-1}}H\big),
\label{eq:nuj}
\end{equation} 
the jump in gradient is
\begin{eqnarray*}
[\nu u^*_x]_j
&=&\nu_{j+1}\frac{U_{j+1}-U_j}H-\nu_j\frac{U_j-U_{j-1}}H
\\&=&\frac1H(\nu_{j+1}U_{j+1}-\nu_{j+1}U_j+\nu_jU_{j-1}-\nu_jU_j)
\\&=&\frac1H\left(\nu|_{X_{j+1}^-}u^*|_{X_{j+1}^-}
-\nu|_{X_{j}^+}u^*|_{X_{j}^+}
+\nu|_{X_{j-1}^+}u^*|_{X_{j-1}^+}
-\nu|_{X_{j}^-}u^*|_{X_{j}^-}\right)\,,
\end{eqnarray*}
which is the required right-hand side of~\eqref{eq:ccj} for coupling parameter \(\gamma=0\) (as \(C(0)=1\)).
Hence, the piecewise linear fields~\eqref{eq:equil}, with \(\alpha=\gamma=0\), are equilibria for all~\Uv, and thus form an \(N\)-D~subspace of equilibria.
\end{proof}

The spectrum comes from the linearised dynamics around each of the equilibria~\EE.
Seek solutions \(u=u^*(x)+\hat u(x,t)\) of the general \pde~\eqref{eq:genpde} where \(\hat u(x,t)\)~denotes a small perturbation to the equilibrium~\eqref{eq:equil}. 
Use \(\hat u_j(x,t)\) as a synonym for~\(\hat u(x,t)\) on the \(j\)th~interval~\(\XX_j\).
Then for parameters \(\alpha=\gamma=0\) and small~\(\hat u\), the \pde~\eqref{eq:genpde} linearises to
\begin{subequations}\label{eqs:lin}%
\begin{equation}
\hat u_t=F'(u^*_x)\hat u_{xx}=\nu(x)\hat u_{xx}\quad\text{on }\tilde\XX; 
\quad\text{that is,}\quad 
\hat u_{jt}=\nu_j\hat u_{jxx}\quad\text{on }\XX_j.
\label{eq:pdelin}
\end{equation}
The coupling conditions~\eqref{eqs:cc} are linear, so they are 
\([\hat u]_j=0\) and \([\nu\hat u_x]_j=\frac1H\big[
(\nu\hat  u)|_{X_{j+1}^-} -(\nu\hat  u)|_{X_{j}^+}
+(\nu\hat  u)|_{X_{j-1}^+} -(\nu\hat  u)|_{X_{j}^-}\big]\);
that is,
\begin{eqnarray}
&&\hat u_{j+1}(X_j)=\hat u_j(X_j)
\quad\text{and }
\label{eq:ccclin}
\\&&
\nu_{j+1}\hat u_{j+1,x}(X_j)-\nu_j\hat u_{jx}(X_j)
\nonumber\\&&{}
=\frac1H\left[
\nu_{j+1}\hat u_{j+1}(X_{j+1}) -\nu_{j+1}\hat u_{j+1}(X_{j})
+\nu_j\hat  u_j(X_{j-1}) -\nu_j\hat  u_j(X_{j})\right].
\qquad\label{eq:ccjlin}
\end{eqnarray}
\end{subequations}
The next lemma certifies that this linearised system is self-adjoint and so we need only seek real eigenvalues in the spectrum.

To be definite, define the Hilbert space~\HH\ to be the set of square integrable, twice differentiable, functions on~\(\tilde\XX\).
Also define its subspace~\LL\ to be those which are additionally \(L\)-periodic.

\begin{lemma}[self-adjoint] \label{thm:sa}
The differential operator appearing in~\eqref{eqs:lin}, namely \(\cL=\nu\DD x{}\) on~\LL\ and subject to~\eqref{eq:ccclin}--\eqref{eq:ccjlin}, is self-adjoint upon using the usual inner product \(\left<v,u\right>:=\int_{\tilde\XX} vu\,dx\)\,.
\end{lemma}

\begin{proof} 
Straightforwardly use integration by parts (remembering that \(\nu\)~is piecewise constant):
\begin{eqnarray*}
\left<v,\cL u\right>
&=&\int_{\tilde\XX}v\nu u_{xx}\,dx
\\&=&\sum_j \left[\nu v u_x-\nu u v_x\right]_{X_{j-1}^+}^{X_j^-}
+\int_{\tilde\XX}\nu v_{xx}u\,dx
\\&=&\sum_j \left[ 
\nu_j v_j(X_j^-) u_{jx}(X_j^-)
-\nu_j u_j(X_j^-) v_{jx}(X_j^-) 
\right.\\&&\quad\left.{}
-\nu_j v_j(X_{j-1}^+) u_{jx}(X_{j-1}^+)
+\nu_j u_j(X_{j-1}^+) v_{jx}(X_{j-1}^+) \right]
+\left<\cL v, u\right>
\\&&(\text{using the continuity~\eqref{eq:ccclin} and }U_j:=u(X_j^\pm),\ V_j:=v(X_j^\pm)\,)
\\&=&\sum_j \left[ 
\nu_j V_j u_{jx}(X_j^-)
-\nu_j U_j v_{jx}(X_j^-) 
\right.\\&&\quad\left.{}
-\nu_j V_{j-1} u_{jx}(X_{j-1}^+)
+\nu_j U_{j-1} v_{jx}(X_{j-1}^+) \right]
+\left<\cL v, u\right>
\\&&(\text{reindexing the last two terms in the sum, }j\mapsto j+1)
\\&=&\sum_j \left[ 
V_j \nu_j u_{jx}(X_j^-)
-U_j \nu_j v_{jx}(X_j^-) 
\right.\\&&\quad\left.{}
- V_{j} \nu_{j+1} u_{j+1,x}(X_{j}^+)
+ U_{j} \nu_{j+1} v_{j+1,x}(X_{j}^+) \right]
+\left<\cL v, u\right>
\\&&(\text{replacing two pairs of terms via coupling~\eqref{eq:ccjlin}})
\\&=&\sum_j \left\{ 
-V_j \tfrac1H\left[ \nu_{j+1}U_{j+1} -\nu_{j+1}U_{j}
+\nu_jU_{j-1} -\nu_jU_j \right]
\right.\\&&\quad\left.{}
+ U_{j}\tfrac1H \left[ \nu_{j+1}V_{j+1} -\nu_{j+1}V_{j}
+\nu_jV_{j-1} -\nu_jV_j \right] \right\}
+\left<\cL v, u\right>
\\&&(\text{cancelling all terms in }U_jV_j)
\\&=&\sum_j \tfrac1H\left\{ 
-V_j\nu_{j+1}U_{j+1} 
-V_j\nu_jU_{j-1} 
+ U_{j} \nu_{j+1}V_{j+1} 
+U_j\nu_jV_{j-1}   \right\}
\\&&\quad{}
+\left<\cL v, u\right>
\\&&(\text{reindexing 2nd and 4th terms in the sum, }j\mapsto j+1)
\\&=&\sum_j 0
+\left<\cL v, u\right> =\left<\cL v, u\right>.
\end{eqnarray*}
Hence,  the linear operator in the linearised system~\eqref{eqs:lin} is symmetric.  
Since \(\cL:\HH\to\HH\) it is self-adjoint in~\HH. 
It can be shown that self-adjointness also holds for Dirichlet and Neumann boundary conditions.
\end{proof}

We turn to determining the spectrum of the general linearised system~\eqref{eqs:lin}: first, the zero eigenvalues; and second, the non-zero eigenvalues.
Because of the \(N\)-D~subspace of equilibria~\EE, the linearised system must have \(N\)~eigenvalues of zero.
Corresponding basis eigenfunctions may be chosen to be 
\begin{equation*}
\phi_j(x)=\max(0,1-|x-X_j|/H)
\end{equation*}
so the equilibria~\eqref{eq:equil} may be written \(u^*=\sum_j \phi_j(x)U_j\)\,.
Incidentally, the localised triangular shape of these basis functions   will be recognised by many as the fundamental ``shape function'' often invoked in the finite element method \cite[e.g.]{OLeary08, Strang2008}.
For the linearised \pde~\eqref{eq:pdelin} any eigenfunction corresponding to an eigenvalue of zero must be linear on each~\(\XX_j\), and the continuity~\eqref{eq:ccclin} then guarantees there are no other eigenfunctions than those identified.
By self-adjointness, there are no generalised eigenfunctions.
Thus the slow subspace of the system~\eqref{eqs:lin} is \(N\)-D, namely~\EE.

For rigorous theory we notionally adjoin the two trivial dynamical equations \(\alpha_t=\gamma_t=0\) to the linearised system~\eqref{eqs:lin}.
Then,  as \(\alpha=\gamma=0\)\,, the equilibria~\eqref{eq:equil} are~\((0,0,u^*(x))\).
Thus strictly there are two extra zero eigenvalues associated with the trivial \(\alpha_t=\gamma_t=0\)\,, and the corresponding slow subspace of each equilibria is \((N+2)\)-D.
Except for issues associated with the domain of validity, for simplicity we do not explicitly include these two trivial dynamical equations nor their eigenvalues in the following, but consider them implicit.

\begin{lemma}[exponential dichotomy] \label{thm:dicho}
Provided function~\(F\) in the \pde~\eqref{eq:genpde} is monotonically increasing with \(F'\geq\nu_{\min}>0\)\,, then  the operator \(\cL=\nu\DD x{}\)  subject to~\eqref{eq:ccclin}--\eqref{eq:ccjlin} in~\LL\ has \(N\)~zero eigenvalues and all other eigenvalues~\(\lambda\) are negative and bounded away from zero by \(\lambda\leq-\nu_{\min}\pi^2/H^2\).
\end{lemma}

\begin{proof} 
The precisely \(N\)~zero eigenvalues are established in the two paragraphs preceding the lemma.
Lemma~\ref{thm:sa} establishes all eigenvalues of~\(\cL\) are real. 
Let \(\lambda\)~be a non-zero eigenvalue and \(v(x)\)~be a corresponding eigenfunction.
Then \(v\perp\EE\) by self-adjointness of~\cL, and, as usual,
\begin{equation*}
(-\lambda)\|v\|^2
=-\lambda\left<v,v\right>
=\left<v,-\lambda v\right>
=\left<v,-\cL v\right>.
\end{equation*}
Decompose the eigenfunction into \(v(x)=\tilde v(x)+\check v(x)\) where \(\check v(x)\)~is piecewise linear, continuous, and satisfies \(\check v(X_j)=v(X_j)\), so that \(\tilde v\)~is also continuous and \(\tilde v(X_j)=0\)\,.
Since \(\check v\in\EE\), so \(\cL\check v=0\) (the check accent on~\(\check v\) is to remind us of its piecewise linear nature).
Consequently, 
\begin{eqnarray*}
-\lambda\|v\|^2
&=&\left<v,-\cL v\right>
=\left<\tilde v+\check v,-\cL\tilde v\right>
\\&=&\left<\tilde v,-\cL\tilde v\right>
+\left<\check v,-\cL\tilde v\right>
\\&=&\left<\tilde v,-\cL\tilde v\right>
\end{eqnarray*}
as, by self-adjointness, \(\left<\check v,-\cL\tilde v\right>=\left<\cL\check v,-\tilde v\right>=\left<0,-\tilde v\right>=0\).
Thus, we proceed to derive the inequality
\begin{eqnarray*}
-\lambda\|v\|^2&=&\left<\tilde v,-\cL\tilde v\right>
=\int_{\tilde\XX}-\nu\tilde v\tilde v_{xx}\,dx
\\&=&\sum_j\left[-\nu\tilde v\tilde v_x\right]_{X_{j-1}^+}^{X_j^-}
+\int_{\tilde\XX}\nu\tilde v_x^2\,dx
\quad(\text{integrating by parts})
\\&=&\sum_j 0
+\int_{\tilde\XX}\nu\tilde v_x^2\,dx
\quad(\text{using }\tilde v(X_j^-)=\tilde v(X_j^+)=0)
\\&\geq&\nu_{\min}\int_{\tilde\XX}\tilde v_x^2\,dx
\quad(\text{as }\nu(x)\geq \nu_{\min}>0).
\end{eqnarray*}
The first consequence of this inequality is that there are no positive eigenvalues~\(\lambda\).

Secondly, relate this inequality to the spatially homogeneous problem.
Let \(\cL_1=\DD x{}\)~denote the linear operator with coupling conditions appearing in~\eqref{eqs:lin} for the special case of \(\nu(x)=\nu_j=1\) for all~\(x\) and~\(j\).
Then, by the reverse argument to that of the previous paragraph,
\begin{equation*}
\int_{\tilde\XX}\tilde v_x^2\,dx
=\cdots=\left<\tilde v,-\cL_1\tilde v\right>
=\cdots=\left< v,-\cL_1 v\right>.
\end{equation*}
But, by the Rayleigh--Ritz theorem, the smallest magnitude, non-zero, eigenvalue~\(\lambda_1\) of~\(\cL_1\) satisfies \(-\lambda_1=\min_{w\perp\EE} \left<w,-\cL_1w\right>/\|w\|^2\), and so \(\left< v,-\cL_1 v\right> \geq -\lambda_1\|v\|^2\).
Hence the inequalities give \(-\lambda\|v\|^2\geq \nu_{\min}(-\lambda_1)\|v\|^2\).
By the next Lemma~\ref{lem:hs}, \(-\lambda_1\geq\pi^2/H^2\), and so all eigenvalues satisfy \(-\lambda\geq\nu_{\min}\pi^2/H^2\) as required.
\end{proof}

\begin{lemma}[spatially homogeneous spectrum] \label{lem:hs}
The non-zero eigenvalues of the differential operator \(\DD x{}\) with coupling conditions~\eqref{eq:ccclin}--\eqref{eq:ccjlin} in~\HH\ and when \(\nu_j=1\) all satisfy \(\lambda\leq-\pi^2/H^2\).
\end{lemma}

\begin{proof} 
For the spatially homogeneous problem~\eqref{eqs:lin}, set \(\nu(x)=\nu_j=1\)\,. 
Seeking solutions \(e^{\lambda t}v(x)\) leads to the \ode\ \(\lambda v=v''\) on~\(\tilde\XX\).
As a constant coefficient \ode, and for eigenvalues \(\lambda=-\kappa^2/H^2\) for some nondimensional wavenumber~\(\kappa\geq0\) to be determined, its general solutions are of the form \(v_j=A_j\cos\kappa\xi_j+B_j\sin\kappa\xi_j\) for coefficients~\(A_j\) and~\(B_j\) determined by the coupling conditions~\eqref{eq:ccclin}--\eqref{eq:ccjlin}. 
Consequently, the spatial derivative is \(v_{jx}=-\frac{A_j\kappa}H\sin\kappa\xi_j+\frac{B_j\kappa}H\cos\kappa\xi_j\).

Let's consider the spatial map \((A_j,B_j)\mapsto(A_{j+1},B_{j+1})\).
Continuity~\eqref{eq:ccclin} at \(x=X_j\) ($\xi_j=1$) requires 
\begin{equation*}
A_{j+1}=A_j\cos\kappa+B_j\sin\kappa=cA_j+sB_j\,,
\end{equation*}
where, for brevity in this proof, let \(c:=\cos\kappa\) and \(s:=\sin\kappa \)\,.
The derivative jump~\eqref{eq:ccjlin} at \(x=X_j\) requires
\begin{equation*}
\frac\kappa H B_{j+1}+A_j\frac\kappa H s-B_j\frac\kappa H c
=\frac{C(\gamma)}H\left[cA_{j+1}+sB_{j+1} -2A_{j+1} +A_j \right],
\end{equation*}
where we include the factor~\(C(\gamma)\) for a little more generality;
that is,
\begin{equation*}
(c-2)CA_{j+1}+(Cs-\kappa)B_{j+1}+(C-s\kappa)A_j+c\kappa B_j=0\,.
\end{equation*}
Dividing by~\(C\) and setting \(\kappa'=\kappa/C\)\, gives the equivalent
\begin{equation*}
(c-2)A_{j+1}+(s-\kappa')B_{j+1}+(1-s\kappa')A_j+c\kappa'B_j=0\,.
\end{equation*}
Considering together the two mapping equations, this spatial map has solutions \((A_j,B_j)\propto \mu^j\) for some multiplier~\(\mu\) given by vanishing determinant
\begin{eqnarray}&&
\det\begin{bmatrix} \mu-c&-s\\
\mu (c-2)+1-s\kappa'&\mu(s-\kappa')+c\kappa' \end{bmatrix}=0\,;
\nonumber
\\&&\text{that is,}\quad
\mu^2-2\frac{s-c\kappa'}{s-\kappa'}\mu+1=0\,.
\end{eqnarray}
Hence the two possible multipliers of the spatial map are 
\begin{equation*}
\mu=\beta\pm\sqrt{\beta^2-1}\quad\text{for }\beta=\frac{s-c\kappa'}{s-\kappa'}\,.
\end{equation*}
Consequently, \(|\beta|>1\) is not possible as then there would be two (real) multipliers: one with magnitude greater than one, representing structures growing exponentially quickly to the right; and one with magnitude less than one, representing structures growing exponentially quickly to the left.
The only allowable cases occur for \(|\beta|\leq1\) when the multipliers are complex of magnitude \(|\mu|=1\), and so characterise periodic structures in space.
Since \(\kappa'-s=\kappa/C(\gamma)-\sin\kappa\geq0\)\,, the requirement \(|\beta|\leq1\) becomes \(s-\kappa'\leq c\kappa'-s\leq\kappa'-s\)\,; that is, \(2s-\kappa'\leq c\kappa'\leq\kappa'\).
The right-hand inequality is always satisfied as \(c=\cos\kappa\)\,, but the left-hand inequality requires \(2s\leq(1+c)\kappa'\), that is, \(s/(1+c)\leq\kappa'/2\)\,.
Recalling \(s=\sin\kappa\) and \(c=\cos\kappa\), this requirement becomes
\begin{equation}
\tan\frac\kappa2\leq\frac\kappa{2C}\,.
\label{eq:kappa}
\end{equation}
\begin{figure}
\centering
\begin{tabular}{cc}
\tikzsetnextfilename{spectrum}
\begin{tikzpicture}
\begin{axis}[ xlabel={$\kappa$}
  ,axis x line=middle,axis y line=middle
  ,thick,grid,samples=101
  ,ymin=-1,ymax=8 ,domain=0:17
  ,xtick={3.1416,6.2832,9.4248,12.5664,15.7080}
  ,xticklabels={$\pi$,$2\pi$,$3\pi$,$4\pi$,$5\pi$}
  ,legend pos=north west
  ] 
\addplot [blue,no marks] {tan(deg(x)/2)}; 
\addlegendentry{$\tan(\kappa/2)$};
\addplot [dashed,red,no marks,samples=2] {x/2}; 
\addlegendentry{$\kappa/2$};
\addplot [brown,only marks,mark=*] {0-99*(x/2<tan(deg(x)/2))};
\addlegendentry{wavenumber};
\end{axis}
\end{tikzpicture}
&
\parbox[b]{0.4\linewidth}{\caption{\label{fig:leaspec}%
The linearised ($\gamma=0$, \(C=1\)) spectrum determined by the spectral requirement~\eqref{eq:kappa}.
The thick lines along the $\kappa$-axis indicate regions of spatial wavenumbers for which the corresponding spatial structures are bounded for all space (the essential spectrum).}}
\end{tabular}
\end{figure}
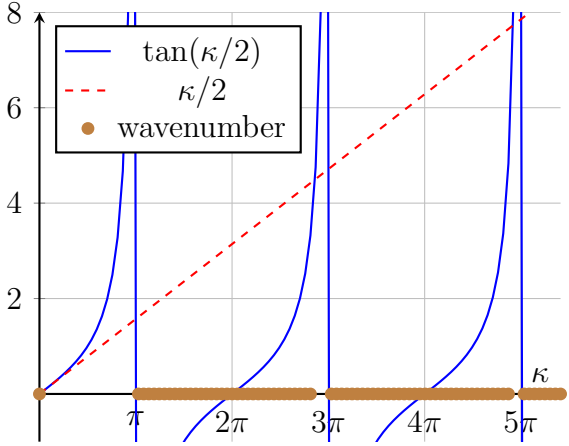%
The specific boundary conditions on the finite macroscale domain~\XX\ then constrain the allowable~\(\kappa\) to a discrete, countably infinite, set of~\(\kappa\) satisfying inequality~\eqref{eq:kappa}.
As illustrated by Figure~\ref{fig:leaspec}, for the specific case of the linearised problem~\eqref{eqs:lin} for which \(C(0)=1\)\,, there is a useful spectral gap because the smallest allowable nonzero nondimensional wavenumber is \(\kappa=\pi\)\,.
Hence the smallest magnitude nonzero eigenvalue is\({}\leq-\pi^2/H^2\) as required.
\end{proof}

The reason to include \(C(\gamma)\)~in the proof is to comment on the  linearisation about another subspace of equilibria.
As well as the piecewise linear equilibria~\EE\ at \(\gamma=\alpha=0\)\,, another subspace of equilibria is \(u={}\)constant on~\XX\ for nonlinearity \(\alpha=0\) but now for arbitrary coupling parameter~\(\gamma\).
The linearisation about this set of equilibria is also the system~\eqref{eqs:lin} but with \(\nu(x)=F'(0)\), constant, and with factor~\(1/H\) in the coupling~\eqref{eq:ccjlin} replaced by~\(C(\gamma)/H\).
The proof of Lemma~\ref{lem:hs} also applies to this case.
Inequality~\eqref{eq:kappa} then gives allowed wavenumbers for general~\(\gamma\).
As coupling parameter~\(\gamma\) varies from zero to one, the factor~\(C(\gamma)\) varies from one to zero, and so the denominator~\(C\) in inequality~\eqref{eq:kappa} increases the slope of the straight line of Figure~\ref{fig:leaspec}.
Thus the set of allowed wavenumbers increases with coupling~\(\gamma\), and, in particular, the spectral gap~\((0,\pi)\) between the slow and the fast modes fills up with the slow modes.
It is in this manner that the continuum of allowed wavenumbers is restored in the fully coupled \pde\ over the whole domain~\XX, as the coupling parameter~\(\gamma\) varies from zero to one.

\begin{theorem}[slow manifold] \label{thm:gsm}
Consider the nonlinear \pde~\eqref{eq:genpde} on domain~\(\tilde\XX\) with coupling conditions~\eqref{eqs:cc} and preconditions as recorded earlier.  
\begin{enumerate}
\item In an open domain~\cE\ containing the subspace~\EE\ there exists a slow manifold,
\begin{equation}
u=u(x,\Uv,\gamma,\alpha)
\quad\text{such that}\quad
\frac{d\Uv}{dt}=\gv(\Uv,\gamma,\alpha)
\label{eq:gensm}
\end{equation}
(generally one order less smooth than that of~\(F\) and~\(G\)).
\item\label{thm:gsm:emergent} This slow manifold is emergent in the sense that for all solutions~\(u(x,t)\) of~\eqref{eq:genpde} and~\eqref{eqs:cc}, that stay in~\cE, there exists a solution~\(\Uv(t)\) of~\eqref{eq:gensm} such that \(u(x,t)=u(x,\Uv(t),\gamma,\alpha)+\Ord{e^{-\mu t}}\) for decay rate \(\mu\approx\nu_{\min}\pi^2/H^2\).

\item Given two smooth functions \(\tilde u(x,\Uv,\gamma,\alpha)\) and~\(\tilde\gv(\Uv,\gamma,\alpha)\) for the governing equations~\eqref{eq:genpde} and~\eqref{eqs:cc}, evaluated at \(u=\tilde u\) such that \(d\Uv/dt=\tilde\gv\), having residuals~\Ord{\gamma^p+\alpha^q} as \((\gamma,\alpha)\to\vec 0\)\,,  then the slow manifold is 
\begin{equation*}
u=\tilde u(x,\Uv,\gamma,\alpha)+\Ord{\gamma^p+\alpha^q}
\quad\text{such that }
\frac{d\Uv}{dt}=\tilde\gv(\Uv,\gamma,\alpha)+\Ord{\gamma^p+\alpha^q}.
\end{equation*}
\end{enumerate}
\end{theorem}

\begin{proof} 
The preconditions for the centre manifold theorems of \cite{Haragus2011} [Chapter~2] hold.
We consider twice differentiable, \(L\)-periodic, square integrable functions on~\(\tilde\XX\) which forms the requisite Hilbert spaces.
The self-adjoint, linearised operator~\eqref{eqs:lin} of diffusion on a finite spatial domain forms an analytic semigroup \cite[Remark~2.18, e.g.]{Haragus2011}, and the functions \(F\) and~\(G\) of the \pde~\eqref{eq:genpde} are assumed smooth to thus satisfy Hypothesis~2.1 and~2.7 of \cite{Haragus2011}.
Lemma~\ref{thm:dicho}, under the proviso that \(F'\geq\nu_{\min}>0\), establishes the Spectral Decomposition Hypothesis~2.4 of \cite{Haragus2011}.
\begin{enumerate}
\item Theorem~2.9 of \cite{Haragus2011} then establishes that for each point of~\EE\ (parametrised by~\Uv) a \emph{local centre manifold}~\(\cM_{\Uv}\) exists in some neighbourhood~\(\cE_{\Uv}\) in the \((\Uv,\gamma,\alpha)\)-space.
Because the centre eigenvalues are all zero (Lemma~\ref{thm:dicho}), they are more precisely called \emph{local slow manifolds}.
Setting \(\cM:=\bigcup_{\Uv}\cM_{\Uv}\) and domain \(\cE:=\bigcup_{\Uv}\cE_{\Uv}\) the slow manifold~\cM\ exists in the domain~\cE\ (containing~\EE) as required.

\item The unstable spectrum is empty (Lemma~\ref{thm:dicho}), so Theorem~3.22 of \cite{Haragus2011} applies to establish the exponentially quick emergence of the slow manifold to all solutions that remain within~\cE\ for all time.
The rate of attraction to the slow manifold in~\cE\ is estimated by the linearised rate at~\EE\ by continuity in perturbations \cite[\S11.3, e.g.]{Roberts2014a}.

\item Under corresponding preconditions, Proposition~3.6 of \cite{Potzsche2006} proves that if an approximation to the slow manifold~\eqref{eq:gensm} gives residuals of the system's equations which are zero to some order, then the slow manifold is approximated to the same order of error.
Here introduce parameter~\(\epsilon\) and set \(\gamma=c\epsilon^q\) and \(\alpha=a\epsilon^p\).
Then regard quantities as a Taylor series in~\(\epsilon\) with coefficients parametrised by~\((\Uv,c,a)\).
Also, the process \(\epsilon\to 0\) implies \((\gamma,\alpha)\to\vec 0\)\,.
By supposition, the given~\(\tilde u\) and~\(\tilde\gv\) have residuals \(\Ord{\gamma^p+\alpha^q}=\Ord{c^p\epsilon^{pq}+a^q\epsilon^{pq}}=\Ord{\epsilon^{pq}}\) as \(\epsilon\to 0\)\,.
By Proposition~3.6 of \cite{Potzsche2006}, \(\tilde u\) and~\(\tilde\gv\) approximate the slow manifold to errors~\Ord{\epsilon^{pq}}, and hence the errors are~\Ord{\gamma^p+\alpha^q}.

\end{enumerate}
The more wide ranging theorems of \cite{Aulbach96, Aulbach99, Aulbach2000} could also be invoked to establish this theorem.
\end{proof}

The evolution equation~\eqref{eq:gensm}, evaluated at full coupling, \(d\Uv/dt=g(\Uv,1,\alpha)\), is the in-principle exact closure for a discretisation of the dynamics of the nonlinear \pde~\eqref{eq:genpde}.

\subsection{The slow manifold of wave-like PDEs}
\label{sec:smwpde}

Although this article's scope is the spatial discretisation, or dimensional reduction, of reaction-advection-diffusion \pde{}s~\eqref{eq:genpde}, much of the theory usefully applies to the spatial discretisation of wave-like \pde{}s in the form
\begin{equation}
u_{tt}=F(u_x)_x+\alpha G(x,u,u_x)
\label{eq:wavpde}
\end{equation}
on a domain~\XX, and for smooth functions~\(F\) and~\(G\) as before.
This subsection comments on the similarities and differences of the theoretical support for such wave systems.

Partition space as above and apply the coupling conditions~\eqref{eqs:cc}.
Then, for \(\alpha=\gamma=0\)\,, the subspace~\EE\ of piecewise linear equilibria of Lemma~\ref{thm:equil} still exists.
Upon linearisation about each of these equilibria, the spatial differential operator \(\cL=\nu\DD x{}\) on~\(\tilde\XX\) remains self-adjoint (Lemma~\ref{thm:sa}).
The exponential dichotomy of the operator~\cL\ (Lemmas~\ref{thm:dicho} and~\ref{lem:hs}) still applies, namely that there are \(N\)~eigenvalues of zero, and the others are\({}\leq-\nu_{\min}\pi^2/H^2\).
So far, diffusion-based systems and wave-based systems are the same.

The differences in theoretical support start with Theorem~\ref{thm:gsm}.
The reason for the differences are that the eigenvalues of the right-hand side operator~\cL\ are the square of the eigenvalues of the linearisation of the wave \pde~\eqref{eq:wavpde}: seeking waves of frequency~\(\omega\) then \(|\omega|=\sqrt{-\lambda}\) and all frequencies~\(\omega\) are real as all eigenvalues are\({}\leq0\)\,.
Here the slow manifold dichotomy is now between slow waves with near zero frequency, separated from fast waves with frequencies\({}\geq\sqrt{\nu_{\min}}\pi/H\).
Such subcentre slow  manifolds are ubiquitous in geophysical applications.
However, much less is known rigorously about subcentre slow  manifolds: even their existence is problematic \cite[]{Lorenz87}.
Nonetheless, based upon recursively constructing coordinate transforms to a normal form \cite[Chap.~13]{Cox93a, Cox93b, Roberts2014a} the following `backwards' conjecture \cite[e.g.]{Grcar2011} is indicated for the wave \pde~\eqref{eq:wavpde}.
Parts of this conjecture for waves correspond to Theorem~\ref{thm:gsm} for dissipative systems.

\begin{conjecture}
Specify any order of error~\Ord{\gamma^p+\alpha^q}.
Then there exists a (multinomial) coordinate transformation and a (multinomial) \pde\ system in the new variables~\((\Uv,V(x))\) of the form
\begin{equation}
u=u(x,\Uv,V,\gamma,\alpha),\quad
\dd t\Uv=\gv(\Uv,V,\gamma,\alpha),\quad
\DD tV=H(x,\Uv,V,\gamma,\alpha)V,
\label{eq:wavct}
\end{equation}
such that  in the \(u\)-space the corresponding dynamics is the same as the \pde~\eqref{eq:wavpde} to an error~\Ord{\gamma^p+\alpha^q}, and \(u(x,\Uv,0,\gamma,\alpha)\) is tangent to the subspace~\EE\ at \(\gamma=\alpha=0\)\,.
\emph{(A difference with Theorem~\ref{thm:gsm}.3 is that here we construct a `nearby' approximating system and then base results on that.)}
\begin{enumerate}
\item Let \cE\ denote a \(u\gamma\alpha\)-domain in which the coordinate transform~\eqref{eq:wavct} is a diffeomorphism containing~\EE, then \(V=0\) is an exact slow manifold of the dynamics of~\eqref{eq:wavct}: that is, 
\begin{equation}
u=u(x,\Uv,0,\gamma,\alpha)\quad \text{such that}\quad
\dd t\Uv=\gv(\Uv,0,\gamma,\alpha).
\label{eq:wavsm}
\end{equation}
\emph{(A difference with Theorem~\ref{thm:gsm}.1 is that here we only know that there are nearby systems which have slow manifolds, but like the dissipative case, such a nearby system does possess an exact low-dimensional closure.)}

\item Solutions near, but off the slow manifold with \(V\neq 0\), generally evolve differently (due to wave-wave forcing of mean flow):
\begin{equation*}
\dd t\Uv=\gv(\Uv,0,\gamma,\alpha)+\Ord{\|V\|^2}.
\end{equation*}
\emph{(A difference with Theorem~\ref{thm:gsm}.\ref{thm:gsm:emergent} is that here the slow manifold is not exponentially attractive and instead generally acts as a dynamical centre for nearby dynamics.)}

\end{enumerate}
\end{conjecture}

Consequently, we contend that the methodology developed here for constructing and using spatially discrete, finite dimensional, models of dissipative \pde{}s may be also usefully applied to wave-like \pde{}s~\eqref{eq:wavpde}.

\section{Nonlinear modelling of Burger's PDE}
\label{sec:nonlin}

This section uses Burgers' \pde~\eqref{eq:burgers} as an example of the construction of a slow manifold discrete model.
Burgers' \pde~\eqref{eq:burgers} is in the class~\eqref{eq:genpde} addressed by the theory of Section~\ref{sec:lin} and so Theorem~\ref{thm:gsm} assures us a slow manifold model exists.

Proposition~3.6 by \cite{Potzsche2006} underlies the construction as it asserts the order of error of an approximation is the same as the order of error of the residuals of the governing equations.
Given the existence of a slow manifold $u=u(x,\Uv)$ such that $\dot{\Uv}=\gv(\Uv)$, and implicitly a function of coupling~\(\gamma\) and nonlinearity~\(\alpha\), we rewrite Burgers' \pde~\eqref{eq:burgers} in the form 
\begin{equation}
\cR(u,\gv)=0 \quad\text{with residual }
 {\cR}({u},\gv) := -\D {\Uv}{{u}}\cdot\gv+\nu\DD x{u}-\alpha{u}{u}_{x} \,.
\label{eq:man}
\end{equation}

The initial approximation to the slow manifold is, in terms of the local space variable~\(\xi_j=(x-X_{j-1})/H\) defined by~\eqref{eq:equil}, the piecewise linear field
\begin{equation}
u^0=\sum_{j=1}^{N}\chi_j(x)\big[(1-\xi_j)U_{j-1}+\xi_jU_j\big]
\quad\text{where }
\chi_j(x)=\begin{cases}
1&\text{if }x\in\XX_j\,,\\
0&\text{if }x\not\in\XX_j\,.
\end{cases}
\label{eq:u0}
\end{equation}
We seek the slow manifold for the coupled and nonlinear dynamics in a multivariate power series in corresponding parameters~\(\gamma\) and~\(\alpha\).
But to simplify the algebraic construction process we follow the approach of \cite{Jarrad2001} and introduce one ordering parameter \(\varepsilon=\sqrt{\gamma^2+\alpha^2}\) and label terms depending upon their order in~\(\varepsilon\).
For example, a term in~\(\gamma^p\alpha^q\) is termed of order~\(\varepsilon^{p+q}\).
Then we seek expressions for the slow manifold in the asymptotic series
\begin{eqnarray}
  {u}(x,\Uv,\gamma,\alpha) \sim \sum_{n=0}^{\infty} {u}^{n}(x,\Uv,\gamma,\alpha),
&&
  \gv(\Uv,\gamma,\alpha) \sim \sum_{n=1}^{\infty}\gv^{n}(\Uv,\gamma,\alpha),
\label{eq:power}
\end{eqnarray}
where \(u^n\) and~\(\gv^n\) are of order~\(n\) in the order parameter~\(\varepsilon\).
The partial sums of these series are
\begin{eqnarray}
   {u}^{\langle n\rangle}(x,\Uv,\gamma,\alpha) 
   := \sum_{p=0}^{n}{u}^{p},
&&
  \gv^{\langle n\rangle}(\Uv,\gamma,\alpha) 
  := \sum_{p=1}^{n}\gv^{n}.
\end{eqnarray}
Then ${u}={u}^{\langle n\rangle}+\Ord{\varepsilon^{n+1}}$ and
${u}^{\langle n\rangle}={u}^{\langle n-1\rangle}+u^n$; 
likewise for $\gv^{\langle n\rangle}$ and~$\gv^n$.
Substituting these into the governing \pde~\eqref{eq:man} and rearranging we deduce
\begin{equation*}
 {\cR}\left({u}^{\langle n\rangle},\gv^{\langle n\rangle}\right) = 
  -\D {\Uv}{{u^0}}\cdot\gv^n +\nu\DD x{u^n}
  +{\cR}\left({u}^{\langle n-1\rangle},\gv^{\langle n-1\rangle}\right)
+\Ord{\varepsilon^{n+1}}\,.
\end{equation*}
Hence to require the residual \({\cR}\big({u}^{\langle n\rangle},\gv^{\langle n\rangle}\big) = \Ord{\varepsilon^{n+1}}\), the process is to iteratively solve
\begin{equation}
\nu\DD x{u^n} = \D {\Uv}{{u^0}}\cdot\gv^n-{\cR}\big({u}^{\langle n-1\rangle},\gv^{\langle n-1\rangle}\big)
\label{eq:order-n}
\end{equation}
for corrections~\(u^n\) to the subgrid field and corrections~\(\gv^n\) to the slow manifold closure of the evolution.

The computer algebra code listed in the Ancillary Material (Appendix~\ref{sec:aa}) confirms the following algebraic summary.

\paragraph{First order approximation}

Obtain the first approximation by solving~\eqref{eq:order-n} for the case \(n=1\) given the initial subspace approximation~\eqref{eq:u0}.
Defining the backward difference operator $\nabla:=1-\sigma^{-1}$, equation~\eqref{eq:order-n} becomes
\begin{eqnarray*}
\nu\DD x{u^1} = \sum_{j=1}^{N}\chi_j(x)\left\{\big[\xi_j+(1-\xi_j)\sigma^{-1}\big]g^1_{j}
   +\alpha\left(\big[\xi_j+(1-\xi_j)\sigma^{-1}
\big]U_j\right)\frac{1}{H}\nabla U_j\right\}.
\end{eqnarray*} 
Spatially integrating twice gives
\begin{eqnarray*}
\nu{u^1} & = &  \sum_{j=1}^{N}\chi_j(x)\left\{d_j+H\xi_j c_j+\frac{H^2}{6}\big[\xi_j^3+(1-\xi_j)^3\sigma^{-1}\big]g^1_{j}
\right.\nonumber\\
&& \qquad\left.{}
+\frac{\alpha H}{6}\left(\big[\xi_j^3+(1-\xi_j)^3\sigma^{-1}\big]U_j\right)\,\nabla U_j\right\}.
\end{eqnarray*}
\begin{itemize}
\item The inter-element continuity condition~\eqref{eq:ccc} requires that ${u^n}(X_j,\Uv)=0$ for $n=1,2,3,\ldots$,
because ${u^0}(X_j,\Uv)=U_j=u(X_j,\Uv)$. 
Hence, we solve for~$d_j$ at $\xi_j=0$ and~$c_j$ at $\xi_j=1$, giving
\begin{eqnarray}
\nu{u^1} & = &  \sum_{j=1}^{N}\chi_j(x)\left\{\frac{H^2}{6}\big[\xi_j^3+(1-\xi_j)^3\sigma^{-1}
-\xi_j\nabla-\sigma^{-1}\big]g^1_j
\right.\nonumber\\&& \qquad\left.{}
+\frac{\alpha H}{6}\left(\big[\xi_j^3+(1-\xi_j)^3\sigma^{-1}-\xi_j\nabla-\sigma^{-1}\big]U_j\right)\,\nabla U_j\right\}
\nonumber\\
& = &  \sum_{j=1}^{N}\chi_j(x)\left\{\frac{H^2}{6}\cI_1g^1_j
+\frac{\alpha H}{6}(\cI_1U_j)\,\nabla U_j\right\},
\qquad\label{eq:u1g1}
\end{eqnarray}
where it is convenient to introduce interpolation operators ${\cI}_0(\xi_j):=\xi_j+(1-\xi_j)\sigma^{-1}$ and
${\cI}_1(\xi_j):=\xi_j^3+(1-\xi_j)^3\sigma^{-1}-\xi_j\nabla-\sigma^{-1}$ (observe that ${\cI}_1''=6{\cI}_0/H^2$).

\item The inter-element smoothness condition~\eqref{eq:ccj} determines the value of~$g^1_j$.
Substitute~\(u^1\) from~\eqref{eq:u1g1} into~\eqref{eq:ccj} and we require
\begin{equation*}
-\frac{\nu\gamma}{H}\delta^2 U_j = -H\left(1+\frac{1}{6}\delta^2\right)g^1_j 
-\frac{\alpha}{3}\left(U_j\mu\delta U_j+\mu\delta U_j^2\right).
\end{equation*}
For this to be satisfied we set
\begin{equation}
g^1_j=S\left[\frac{\nu\gamma}{H^2}\delta^2 U_j
-\frac{\alpha}{3H}U_j\mu\delta U_j-\frac{\alpha}{3H}\mu\delta U_j^2
\right],
\label{eq:g1}
\end{equation}
where operator $S:=(1+\delta^2/6)^{-1}$.
\end{itemize}
Combining these with the initial approximation gives the slow manifold
\begin{subequations}\label{eqs:}%
\begin{eqnarray}&&
 {u} = \sum_{j=1}^{N}\chi_j(x)\left\{{\cI}_0 U_j+\frac{H^2}{6\nu}{\cI}_1g^1_j
+\frac{\alpha H}{6\nu}{\cI}_1 U_j\,\nabla U_j\right\}+\Ord{\varepsilon^2},
\\&&
  \dot{U}_j  = S\left[\frac{\nu\gamma}{H^2}\delta^2 U_j
-\frac{\alpha}{3H}U_j\mu\delta U_j-\frac{\alpha}{3H}\mu\delta U_j^2
\right]+\Ord{\varepsilon^2}.
\label{eq:dudt1}
\end{eqnarray}
\end{subequations}
Apart from the nonlocal operator~$S$, this discrete closure~\eqref{eq:dudt1} is just the mixture model~\eqref{eq:mixture}
with $\theta=\frac{2}{3}$. 
This parameter value is exactly the critical value predicted by \cite{Fornberg73} to be necessary  for the stability of numerical integration of the mixture model with $\nu=0$ and $\alpha=1$.

\paragraph{Connection to a cubic spline}
An intriguing property of the operator~\(S=(1+\delta^2/6)^{-1}=6(\sigma+4+\sigma^{-1})^{-1}\) is that it is precisely the operator found in constructing a cubic spline interpolation through equi-spaced data. For example,
if the general cubic spline for the $j$th~interval is specified as 
\(S_j(x)=a_jH^3\xi_j^3/6+b_jH^2\xi_j^2/2+c_jH\xi_j+d_j\), 
then its second derivative at the left-hand end of the interval
is given by $b_j=S\delta^2 d_j/H^2$,
and the corresponding first and third derivatives by
$c_j=\nabla\sigma d_j/H-H(2+\sigma)b_j/6$
and $a_j=\nabla\sigma b_j/H$, respectively 
 \cite[e.g.]{BurdenFaires85}.
As $d_j=U_{j-1}$ in our example, comparison with the first-order approximation derived above reveals that 
the holistic approach ensures a cubic spline approximation when $\alpha=0$ and $\gamma=1$.

\paragraph{Higher order approximations}

Higher order terms in the asymptotic series for~${u}$  may be systematically computed by iteratively solving equation~\eqref{eq:order-n} after having first computed~$\gv^n$ (as implemented in the computer algebra code of the Ancillary Material, Appendix~\ref{sec:aa}). 
The solvability condition determines the latter \cite[]{Jarrad2001}, namely that the right-hand side of equation~\eqref{eq:order-n} must be orthogonal to the null-space of the adjoint of~$\nu\DD x{}$. 
Since the operator is self-adjoint, we isolate the boundary between the~$j$th and $(j+1)$th~intervals with the
triangular finite-element
\begin{eqnarray}
  \hat{v}_0 = \chi_j(x)\xi_j + \chi_{j+1}(x)(1-\xi_{j+1}),
\end{eqnarray}
which satisfies both the continuity condition~\eqref{eq:ccc} and the smoothness condition~\eqref{eq:ccj}
(for $\gamma=0$). 
Now, recall that $u^n(X_j,\vec{U})=0$ for $n\ge 1$,
and hence $[\nu u^n_x]_j=0$ for $n\ge 2$.
Thus, taking the inner product of equation~\eqref{eq:order-n} with~$\hat{v}_0$ gives rise to the solvability condition
\begin{eqnarray}
 HS^{-1}g^n_{j}
-\left< {\cR}\big({u}^{\langle n-1\rangle},\gv^{\langle n-1\rangle}\big), \hat{v}_0 \right>
& = & 0 \quad\text{ for } n=2,3,\ldots\,.
\end{eqnarray}

The higher order advection terms in~$\alpha$ and the interactions between~$\alpha$ and~$\gamma$ rapidly become more complex. 
For example, the $\gamma\alpha$-terms are 
\begin{eqnarray}
	g^{2}_{j} & = & \frac{1}{H}\left\{
  - \frac{1}{10}S\left(U_{j} S\mu\delta U_{j}\right)
  - \frac{1}{6}S\left(U_{j} \mu\delta U_{j}\right)
  + \frac{1}{10}S\left(SU_{j} \mu\delta U_{j}\right)
  \right. \nonumber\\  && \left.{}
- \frac{1}{5}S^{2}\left(U_{j} S\mu\delta U_{j}\right)
  + \frac{13}{30}S^{2}\left(U_{j} \mu\delta U_{j}\right)
  - \frac{1}{15}S^{3}\left(U_{j} \mu\delta U_{j}\right)
  \right. \nonumber\\&& \left.{}
  - \frac{1}{15}S^{3}\mu\delta U_{j}^2
  + \frac{7}{30}S^{2}\mu\delta U_{j}^2
  + \frac{2}{5}U_{j} S\mu\delta U_{j}
  - \frac{11}{30}S\mu\delta U_{j}^2
\right\}
\nonumber\\&&{}
+\Ord{\gamma^2,\alpha^2}.
\qquad
\label{eq:gamma-alpha}
\end{eqnarray}
In contrast, the terms purely in the homotopy parameter~$\gamma$ represent smoothing corrections to the
diffusion; for example,
the coarse dynamics of the diffusion equation ($\alpha=0$)  obey 
\begin{eqnarray}
	\dot{U}_j & = & \frac{\nu\gamma}{H^2}S\delta^2 U _j
+ \frac{\nu\gamma^2}{60H^2}(7-2S)S^2\delta^4 U_j
\nonumber\\
&&{}+ \frac{\nu\gamma^3}{6300H^2}(94-73S+14S^2)S^3\delta^{6}U_j
+ \Ord{\gamma^4}.
\label{eq:diffusion}
\end{eqnarray}
Figure~\ref{fig:diff:spectrum} shows that each additional term in this expansion provides a better approximation to the full continuum dynamics. In particular, since 
$S=\left(1+\frac{1}{6}\delta^2\right)^{-1}\sim 1-\frac{1}{6}\delta^2+\frac{1}{36}\delta^4+\cdots$,
then the first term of equation~\eqref{eq:diffusion} gives 
$S\delta^2 u=H^2 u_{xx}+\Ord{H^4}$ from the relevant Taylor series expansion.
Further, the addition of the second term (for $\gamma=1$) gives
$\left(1-\frac{1}{12}\delta^2\right)\delta^2 u=H^2 u_{xx}+\Ord{H^6}$,
and the addition of the third term gives
$\left(1-\frac{1}{12}\delta^2+\frac{1}{90}\delta^4\right)\delta^2 u=H^2 u_{xx}+\Ord{H^8}$.
In comparison,
observe that $S\mu\delta u\sim\left(1-\frac{1}{6}\delta^2\right)\mu\delta u=H u_x+\Ord{H^5}$;
hence, the conservative term in equation~\eqref{eq:g1} (for $\alpha>0$) is of a higher order approximation than the advective term,
and the latter will require extra corrective terms to provide the same order of accuracy, as demonstrated by
the $\gamma\alpha$-terms of equation~\eqref{eq:gamma-alpha}.

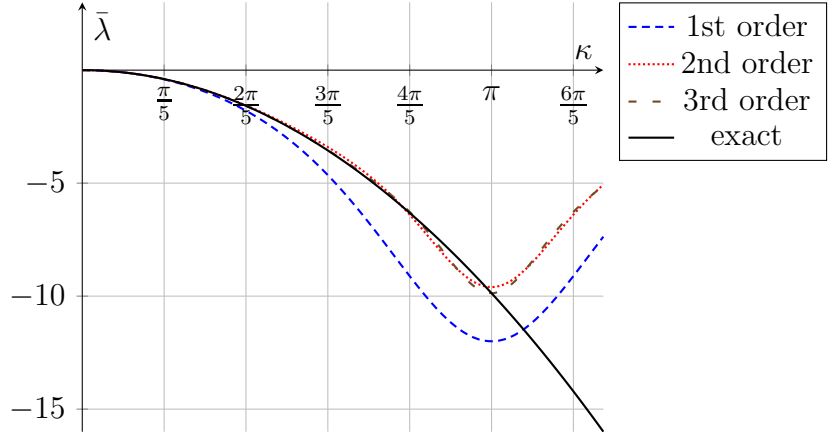
\begin{figure}
\centering
\tikzsetnextfilename{decayRates}
\begin{tikzpicture}
\begin{axis}[ xlabel={$\kappa$}, ylabel={$\bar{\lambda}$}
  ,axis lines=middle
  ,grid,no marks,samples=101
  ,domain=0:4,ymax=3
  ,xtick={0.62832,1.25664,1.88496,2.51328,3.1416,3.76992}
  ,xticklabels={$\frac{\pi}{5}$,$\frac{2\pi}{5}$,$\frac{3\pi}{5}$,$\frac{4\pi}{5}$,$\pi$,$\frac{6\pi}{5}$}
  ,legend pos=outer north east
  ] 
\addplot+[densely dashed,thick] {-6*(1-cos(deg(x)))/(2+cos(deg(x)))}; 
\addlegendentry{1st order};
\addplot+[densely dotted,thick] {-1/5*(96+27*cos(deg(x))-72*cos(deg(x))^2-51*cos(deg(x))^3)/(8+12*cos(deg(x))+6*cos(deg(x))^2+cos(deg(x))^3)}; 
\addlegendentry{2nd order};
\addplot+[loosely dashed,thick] {1/175*(-13824-17010*cos(deg(x))+4050*cos(deg(x))^2+15525*cos(deg(x))^3+8910*cos(deg(x))^4+2349*cos(deg(x))^5)
/(32+80*cos(deg(x))+80*cos(deg(x))^2+40*cos(deg(x))^3+10*cos(deg(x))^4+cos(deg(x))^5)
}; 
\addlegendentry{3rd order};
\addplot+[thick]  {0-(x^2)}; 
\addlegendentry{exact};
\end{axis}
\end{tikzpicture}
\caption{The  non-dimensionalised spectrum 
($\bar{\lambda}=\lambda H^2/\nu$ versus $\kappa=kH$)
of the diffusion equation
($\alpha=0$) for the mode $\tilde{u}(x,t)=e^{\lambda t+ik x}$, 
contrasting the continuum dynamics  against successive discrete, holistic approximations (for $\gamma=1$).}
\label{fig:diff:spectrum}
\end{figure}

\section{Dynamical stability of the discretisation}
\label{sec:numeric}

To investigate the theoretical stability of discretsations to Burgers' equation~\eqref{eq:burgers}, we consider a mostly undisturbed system where $U_j=0$ at all grid-points except for $M$~adjacent, internal points. 
For example, for $M=2$ it suffices to choose $N=M+1=3$ intervals with
outer points fixed at $U_0=U_N=0$. Hence, with the transformation $U_j=\frac{\nu}{\alpha H}V_j$, 
the mixture model~\eqref{eq:mixture} reduces to
\begin{eqnarray*}
\frac{H^2}{\nu}\dot{V}_1 & = & -2V_1+V_2-\frac{(1-\theta)}{2}V_1V_2-\frac{\theta}{4}V_2^2\,,
\\
\frac{H^2}{\nu}\dot{V}_2 & = & V_1-2V_2+\frac{(1-\theta)}{2}V_1V_2+\frac{\theta}{4}V_1^2\,.
\end{eqnarray*}
This reduced system has a stable critical point at $V_1=V_2=0$ with non-dimensionalised eigenvalues 
$\bar{\lambda}=\frac{H^2}{\nu}\lambda=-1, -3$, and 
an unstable critical point at $V_1=-V_2=\frac{12}{2-3\theta}$
with eigenvalues $\bar{\lambda}=\frac{2}{2-3\theta}
\pm\frac{|4-9\theta|}{|2-3\theta|}$. 
Observe that the unstable point is removed to infinity when $\theta=\frac{2}{3}$.
This is exactly the critical value predicted by \cite{Fornberg73} to be necessary (but not always sufficient)
for numerical stability of the mixture model with $\nu=0, \alpha=1$.
Consequently, the corresponding reduction of the holistic model~\eqref{eq:holistic1}, namely
\begin{eqnarray*}
\frac{H^2}{\nu}\dot{V}_1 & = & -4V_1+\frac{11}{4}V_2-\frac{1}{12}V_1^2-\frac{3}{8}V_1V_2-\frac{7}{24}V_2^2\,,
\\
\frac{H^2}{\nu}\dot{V}_2 & = & \frac{11}{4}V_1-4V_2+\frac{7}{24}V_1^2+\frac{3}{8}V_1V_2+\frac{1}{12}V_2^2\,,
\end{eqnarray*}
is unconditionally stable with critical point at $V_1=V_2=0$ and eigenvalues $\bar{\lambda}= -\frac{5}{4}, -\frac{27}{4}$.

Similarly, for $M=3$ consecutive points the mixture model~\eqref{eq:mixture} reduces to
\begin{eqnarray*}
\frac{H^2}{\nu}\dot{V}_1 & = & -2V_1+V_2-\frac{(1-\theta)}{2}V_1V_2-\frac{\theta}{4}V_2^2\,,
\\
\frac{H^2}{\nu}\dot{V}_2 & = & V_1-2V_2+V_3-\frac{(1-\theta)}{2}V_2(V_3-V_1)-\frac{\theta}{4}(V_3^2-V_1^2)\,,
\\
\frac{H^2}{\nu}\dot{V}_3 & = & V_2-2V_3+\frac{(1-\theta)}{2}V_2V_3+\frac{\theta}{4}V_2^2\,.
\end{eqnarray*}
Substitution of $V_1=aV_2$ and $V_3=bV_2$ then leads to
\begin{eqnarray}
V_1=\frac{\mu(4-\mu\theta)}{8+2\mu(1-\theta)}\,, &
V_2=\mu\,, & V_3=\frac{\mu(4+\mu\theta)}{8-2\mu(1-\theta)}\,,
\label{eq:mu-crit}
\end{eqnarray}
where $\mu$ satisfies
\begin{eqnarray}
\mu[\theta(1-\theta)(\theta^2-3\theta+1)\mu^4+16(2\theta^2-4\theta+1)\mu^2-256] = 0\,.
\end{eqnarray}
The trivial critical point corresponding to $\mu=0$ is unconditionally stable.
Observe that the coefficient of~$\mu^4$ vanishes at $\theta=0, 1$ and $\theta_c=\frac{3-\sqrt{5}}{2}$, and that the resulting quadratic equation only possesses real roots
$\mu=\pm 4$ for $\theta=0$ (the purely advective model). However, the critical point~\eqref{eq:mu-crit} is removed to infinity at exactly these roots, 
and so there are no unstable critical points when the $\mu^4$-term vanishes. 
In general, a pair of unstable critical points arise only when $0<\theta<\theta_c$.
Note that this unstable regime excludes the holistic parameter value of $\theta=\frac{2}{3}$.

Turning now to numerical simulation, we assume a $2\pi$-periodic domain with $\nu=1$ and $\alpha=1$ for
convenience. The initial field $u(x,0)=A\sin x$ is integrated at the $N+1$ grid-points $X_j=Hj$ for all $j\in\JJ$,
with spacing $H=\frac{2\pi}{N}$. The integration is performed for a maximum duration of $T=10$, but ceases early at the first sign of either: an instability, detected when $|U_j|>1000$ (denoted by~'\(\times\)'); or a non-monotonic 
irregularity (denoted by~'+'). 
For each number~$N$ of discretised intervals, a search is made over values of~$A$, both positive and negative,
for which instability or irregularity first occurs, as plotted in Figure~\ref{fig:numeric1}.
\begin{figure}
\centering
\includegraphics{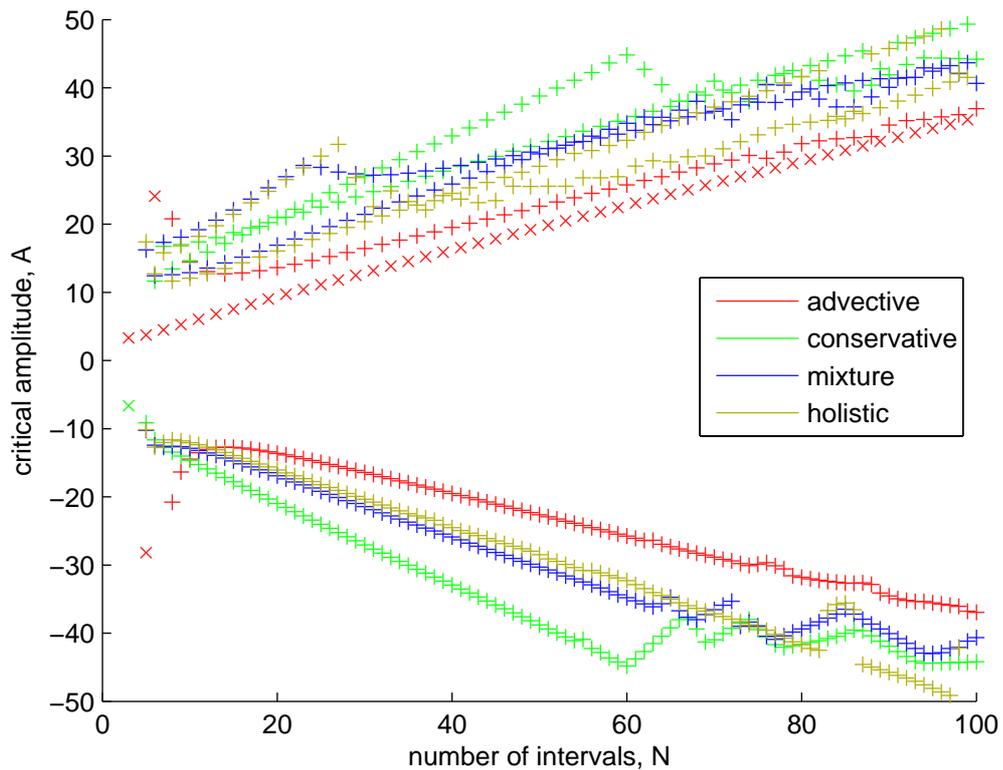}
\caption{Stability of numerical integration of Burgers' \pde, starting from $u(x,0)=A\sin x$. The holistic model
is contrasted with the general mixture model for $\theta=0$ (advective), $\theta=1$ (conservative) and
$\theta=\frac{2}{3}$ (mixture). 
Instability (detected when $|u|>1000$) is denoted by~`\(\times\)', and irregularity
(the presence of non-monotonic modes) is denoted by `+'.}
\label{fig:numeric1}
\end{figure}

The numerical results support the above theoretical results, namely that for $N=3$ intervals ($M=2$ internal points), 
only the advective model ($\theta=0$) and conservative model ($\theta=1$) display instabilities, and
that for $N=4$ none of the simulated models (for $\theta=0$ and $\theta>\theta_c$) show instability, nor irregularity.
Overall, the advective model continues to display instability for odd~$N$, and shows irregularity for even~$N$, 
both of which occur for lower~$|A|$ than the other models.
In contrast, the other models are susceptible to irregularity but not instability, with the critical values
of amplitude~$A$ roughly inversely proportional to the number~$N$ of intervals, 
and thus proprtional to the grid-spacing~$H$. None of the conservative, mixture or holistic models inherently outperforms the others in this measure.
But an advantage of the holistic approach is the rigorous theoretical support, the automatic smooth cubic spline approximation to the out-of-equilibrium subgrid fields, and the automatic derivation of practical sound closures
with unambiguous approximation of the spatial derivatives.

\section{Conclusion}


Holistic discretisation has proved to have a number of attractive properties when applied to the general class of 
diffusive \pde{}s described in Section~\ref{sec:intro}.
In particular, it empowers centre manifold theory, discussed in Section~\ref{sec:lin}, to iteratively refine an initial approximation to a field~$u(x,t)$ whilst
incorporating the dynamics the relevant \pde{}. The resulting approximation is a function of the discrete grid values~$U_j(t)$, the \pde{} parameters, and an introduced homotopy parameter~$\gamma$ that controls
continuity and smoothness. 

The use of a piecewise-linear initial approximation~$u^0$ to~$u$, discussed in Section~\ref{sec:lin}, 
has been shown to be especially effective in conjunction with the holistic approach.
With the inner continuity and smoothness conditions, for instance,
the self-adjointness of the diffusion operator~$\cL$ is preserved under periodic, Dirichlet or Neumann outer boundary conditions. 
This holds for all $\gamma\in[0,1]$, although the usual Neumann condition holds exactly only for $\gamma=1$ and requires modification for $\gamma<1$ (governed by the chosen smoothness condition).
As noted in Section~\ref{sec:nonlin}, also interesting is that
with the addition of the first-order holistic correction~$u^1$, the approximation $u^0+u^1$ for the diffusion equation ($\alpha=0$) 
is an exact cubic spline representation of~$u$ in terms of~$U_j$.
Notably, this cubic spline is defined in terms of the nonlocal operator~$S$, which appears naturally in the holistic derivation of 
the coarse dynamics. 
Holistic analysis of Burgers' equation shows that the derived approximation to~$u$ only approached $C^2$~continuity for high orders of~$\gamma$. 
There remains scope for research into finding an alternative form of the smoothness condition that would lead to $C^n$~smoothness for $\Ord{\gamma^{n+1}}$ approximations.

Another useful facet of holistic discretisation is that it eliminates the ambiguity inherent in choosing appropriate discrete approximations to the spatial derivatives in the \pde{}. 
Section~\ref{sec:nonlin} demonstrated that the induced discretisation follows directly from the \pde{}, as a function of the initial approximation~$u^0$. 
Furthermore, at least in the case of Burgers' equation with a piecewise-linear initial approximation, Section~\ref{sec:numeric} indicated that the holistic discretisation automatically favours numerically stable approximations. 
Indeed, the iterative refinement provided by the holistic procedure acts to improve the order of approximation of spatial derivatives in terms of the discrete grid-spacing~$H$, as discussed in Section~\ref{sec:nonlin}. 
This is somewhat akin to the process of deriving a geometric integration scheme, but applied spatially rather than temporally.
It would be interesting to compare the stability of the fully second-order holistic approximation against the mixture model with correspondingly higher-order spatial derivative approximations.

Finally, another active research direction is the extension of
piecewise-linear holistic discretisation to two or more spatial dimensions, analogous to the results of \cite{Roberts2011a}. 
In general, there is exciting scope for exploring
many more applications of this new approach.

\paragraph{Acknowledgements}
AJR thanks the ARC for partial support of this project through grant DP150102385.

\bibliographystyle{agsm}
\IfFileExists{ajr.sty}
{\bibliography{bibexport,bib,ajr}}
{\bibliography{bibexport}}

\appendix
\section{Ancillary material: computer algebra}
\label{sec:aa}
The following computer algebra code constructs successive slow manifold approximations to Burgers' \pde~\eqref{eq:burgers}.
It is written in the freely available language Reduce.
\footnote{\url{http://www.reduce-algebra.com}}

\verbatimlisting{burgers3.red}

\end{document}